\documentclass[reqno, a4paper, draft]{amsart}
\usepackage{amssymb, amscd, amsfonts,  amsthm}
\usepackage[mathscr]{eucal}
\usepackage{graphicx}
\usepackage[all,poly]{xy} 
 \usepackage[all]{xy} 
 
\newtheorem{theorem}{Theorem}[section]
\newtheorem{lemma}[theorem]{Lemma}
\newtheorem{corollary}[theorem]{Corollary}
\newtheorem{proposition}[theorem]{Proposition}

\theoremstyle{definition}

\newtheorem{remark}[theorem]{Remark}

\newtheorem{problem}[theorem]{Problem}
 
 \newcommand{\restrict}{\,{\mathbin{\vert\mkern-0.3mu\grave{}}}\,}

\newcommand{\remove}[1]{}

\DeclareMathOperator{\Rn}{{\mathbb R^{\it n}}}

\DeclareMathOperator{\Rm}{{\mathbb R^{\it m}}}

\DeclareMathOperator{\maxspec}{\boldsymbol{\mu}}

\DeclareMathOperator{\pos}{\rm pos}

\DeclareMathOperator{\range}{\rm range}
\DeclareMathOperator{\gen}{\rm gen}
\DeclareMathOperator{\bd}{\rm bd}

\DeclareMathOperator{\ellg}{\ell-{\rm group}}

\DeclareMathOperator{\ells}{\ell-{\rm subgroup}}
\DeclareMathOperator{\ellt}{\ell-{\rm group\,\, term}}
\DeclareMathOperator{\elli}{\ell-{\rm ideal}}

 \title[Fans and generators of free abelian $\ell$-groups]
{Fans and generators of free abelian $\ell$-groups}

\author{Daniele Mundici }

\address[D. Mundici]{Department of
Mathematics and Computer Science  ``Ulisse Dini'' \\
University of Florence\\
Viale Morgagni 67/A \\
I-50134 Florence \\
Italy}
\email{ mundici@math.unifi.it }

\date{\today}

\begin{document}


\begin{abstract}    
Let  $t_1,\ldots,t_n$  be  $\ell$-group
terms  in the  variables  $X_1,\ldots,X_m$.
 Let  $\hat t_1,\ldots,\hat t_n$
be their  associated piecewise homogeneous linear
functions.  Let
$G $ be the $\ellg$  generated by $\hat t_1,\ldots,\hat t_n$
in the  free $m$-generator $\ellg$ $\mathcal A_m.$ 
We prove:  (i)  the
problem whether    $G$ is $\ell$-isomorphic to 
  $\mathcal A_n$ is decidable;
 (ii) the 
problem whether    $G$ is $\ell$-isomorphic to $\mathcal A_l$
($l$ arbitrary)
is undecidable;
(iii) for $m=n$, the problem whether  
$\{\hat t_1,\ldots,\hat t_n\}$ is a 
{\it free} generating set
is decidable. 
In view of the 
Baker-Beynon duality,
these  theorems yield 
recognizability and unrecognizability results
for the rational polyhedron   associated
to the $\ell$-group $G$. 
We  make pervasive  
use of  fans and their stellar subdivisions.
\end{abstract}

\maketitle

\section{Foreword and statement of the main result}

The literature on (abelian throughout this paper) $\ellg$ presentations
offers a small  number of decidability/undecidability results, notably
the  coNP-completeness theorem
for the word problem \cite{wei},  and the  
undecidability theorem for the isomorphism problem
of  finitely presented $\ellg$s \cite{glamad},
which follows via the Baker-Beynon duality  from  Markov's
unrecognizability theorem for combinatorial manifolds
 \cite{bak,bey77due}, \cite{chelek, sht}, \cite[Chapter 13]{glahol}.
In this paper  novel
 decidability and undecidability results are proved for 
$\ellg$s ``presented'' by their generating
sets---rather than by principal $\ell$-ideals of free
$\ellg$s. Fans, with their stellar operations,
\cite{ewa,oda}, provide
a  main tool for the proofs.

We assume the reader is well acquainted  with $\ellg$s \cite{and,bkw,gla-book}.  
Every  $\ellg$  term  $t(X_1,\ldots,X_m)$ is canonically interpreted
as a piecewise homogeneous linear
 function  $\hat t\colon \Rm\to \mathbb R$ by setting
$\widehat{X_i} =$  the $i$th coordinate function
$\pi_i\colon\Rm \to \mathbb R$, and then inductively letting
the operation symbols  $\vee,\wedge,+,-$ act
 as the pointwise operations of max, min, sum
and subtraction in $\mathbb R.$ The symbol 0 is interpreted as the
constant zero function  over $\Rn.$

Given  $\ellg$  terms  $t_1,\ldots,t_n$, let
$V=\{X_1,\ldots,X_m\}$ be the union
of the sets of variables occurring in these terms.
It  is no loss of
generality to assume that all the variables of $V$ actually occur in each
$t_j$.
Following \cite{gla-book},  we denote by $\mathcal A_m$   
  the free $\ellg$  over the free generating set 
 $\pi_1,\ldots,\pi_m$.
 Our first two results in this paper are as follows:


\begin{theorem}
\label{theorem:main}
The following problem is decidable:

\smallskip 
\noindent
${\mathsf{INSTANCE}:}$  
$\ell$-group  terms  $t_1,\ldots,t_n$  in
the same variables $X_1,\ldots,X_m$.  
 
\smallskip  
\noindent
${\mathsf{QUESTION}:}$  
Is the $\ell$-subgroup of $\mathcal A_m$  
generated by  \,$\hat t_1,\ldots,\hat t_n$\,\, 
$\ell$-isomorphic to   $\mathcal A_n$?

%


\end{theorem}

 \smallskip 

\begin{theorem}
\label{theorem:main-due}
The following problem is undecidable:

\smallskip  
\noindent
${\mathsf{INSTANCE}:}$    
$\ell$-group  terms  $t_1,\ldots,t_n$  in
the same variables  $X_1,\ldots,X_m$, and an integer $l>0$.

 \smallskip 
\noindent
${\mathsf{QUESTION}:}$  
Is the $\ell$-group  
generated by  \,\,$\hat t_1,\ldots,\hat t_n$  
\,\,$\ell$-isomorphic to   $\mathcal A_l$ ?
\end{theorem}

In Proposition \ref{proposition:hopf} we prove that
every finitely generated free $\ell$-group  $G$  is 
{\it hopfian}, in the sense that every surjective $\ell$-endomorphism
of $G$  is injective, \cite{eva, hir, kos}. Combining this result with
Theorem \ref{theorem:main} we obtain:

\smallskip  
\begin{corollary}
\label{corollary:main-tre}
The following problem is decidable:

 \smallskip 
\noindent
${\mathsf{INSTANCE}:}$  
$\ell$-group  terms  $\,t_1,\ldots,t_n\,$  in
the same variables  $X_1,\ldots,X_n$.

 
\smallskip  
\noindent
${\mathsf{QUESTION}:}$  
Is $\,\{\hat t_1,\ldots,\hat t_n\}\,$
a  {\em free}  generating set of 
the $\ell$-group it generates 
 in $\,\mathcal A_n?$ 
\end{corollary}

In Section \ref{section:final} 
these theorems will be shown to yield
recognizability and unrecognizability results
for the rational polyhedron canonically associated
by the  Baker-Beynon duality
 to the $\ell$-group
generated by $\hat t_1,\dots,\hat t_n$.

%
%
%
%
%
%

\section{Preliminary notions and results}
It is often the case that  the proof of a theorem 
requires a number of  preliminary results  involving 
 many more notions than those quoted in the statement of the theorem itself.
This section and the next one are  devoted to such background
notions and their main properties.
 
 \medskip
\noindent{\it Terminological stipulations.}
Throughout this paper the adjective
``linear'' is understood in the homogeneous sense.
The
number of linear pieces of any piecewise linear map  is finite,
and so is the number of elements of any abstract or geometric complex.
The adjectives  ``decidable, computable, \dots '' stand
for  ``Turing decidable, Turing computable, \dots  ''.
Following \cite{sta}, by a {\it rational polyhedron $P$ in $\Rn$}
we mean a finite union  $\bigcup S_i$ of simplexes  $S_i\subseteq
\Rn$ such that the coordinates of the vertices of each $S_i$
are rational. $P$ need not be convex, nor
connected.  We also let ``$\bd$''  denote boundary. 

\medskip
\noindent
{\it Cones and fans.\,\,} Following \cite{ewa}, an integer
 vector $v\in \mathbb Z^n\subseteq \Rn$ is said to be {\it primitive} if the 
greatest common divisor of its
coordinates is 1.
Given vectors $v_1,\ldots,v_s \in \Rn$ we denote by
$\pos(v_1 ,\ldots, v_s)$ their {\it positive hull}
in $ \Rn$.
In  symbols,
\begin{equation}
 \label{equation:pos}
 \pos(v_1 ,\ldots, v_s)
=\{\lambda_1v_i+\cdots+\lambda_sv_s\mid
\lambda_i\geq 0, \,\,i=1,\ldots,s\}.
\end{equation}
Any set $C$  of the form  $\pos(v_1 ,\ldots, v_s)\subseteq \Rn$
is known as a {\it cone} in $\Rn$.
If   in addition $v_1 ,\ldots, v_s$ are integer vectors in $\mathbb Z^n$
then $C$ is said to be a {\it rational}  cone. 
For $t=1,2,\ldots,n$,
a  {\it $t$-dimensional rational
 simplicial cone} 
 \footnote{Called ``simple cone'' or ``simplex cone''   in \cite{ewa}.}
 in $ \Rn$
  is a set $C \subseteq  \Rn$ of the form 
$$
C=\pos(w_1 ,\ldots, w_t) 
$$
for linearly independent primitive vectors 
$w_1 ,\ldots, w_t \in \mathbb Z^n \subseteq  \mathbb \Rn$.  
The vectors $w_1 ,\ldots, w_t$ are called the primitive
{\it  generating vectors} of $\sigma$.  They are uniquely
determined by $C$. 
By a {\it face } of $C$ we mean the positive hull
of a subset $S$ of $\{w_1 ,\ldots, w_t\}.$
For   completeness we stipulate that
the face of  $C$  determined by the empty
set is the singleton  $\{0\}$.  This is the only
0-dimensional cone in  $\mathbb R^n$.

By a  {\it rational  simplicial fan}, or just a  {\it fan}  in $\mathbb R^n$ 
we mean  a  complex  $\Phi$  of rational simplicial cones
in $\mathbb R^n$: thus $\Phi$ is closed under
taking faces of its cones, and the intersection of any two cones
$C,D\in \Phi$ is a common face of $C$ and $D$.  Note
that the intersection of all cones of $\Phi$ is the singleton 
cone $\{0\}.$
We denote by $|\Phi| \subseteq \mathbb R^{n}$   the
{\it support} of  $\Phi$, i.e., the pointset union of
all cones in $\Phi$.\footnote{Supports of fans are also known as
``rational closed polyhedral cones'' in \cite[p.244]{bey77due},
or ``pointset unions of complexes of rational simplicial cones'',
\cite[p.248]{bey77due}. Fans provide a combinatorial classifier
for toric varieties, \cite{ewa, oda}.}
$|\Phi|$ need not be convex.
Given two fans  $\Sigma$ and $\Delta$ with the same support, 
we say that
  $\Delta$ \emph{is a subdivision of } $\Sigma$  if
each cone of  $\Sigma$ is a union of cones of $\Delta$.

For  some  $t=1,\ldots,n$
let  $C$ be a $t$-dimensional 
rational simplicial cone in $\Rn$, say,
$C=\pos(d_1 ,\ldots, d_t )$, where
where  $d_1 ,\ldots, d_t$ are the
primitive generating vectors of $C$.
 Following \cite[p.146-147]{ewa}, we then say that $C$ is
 {\it regular}
 (``nonsingular''
in \cite[p.15]{oda}, ``primitive'' in 
 \cite[p.246-247]{bey77due})
 if the set $\{d_1 ,\ldots, d_t\}$  can be extended to a basis of
the free abelian group $\mathbb Z^n$ of integer
points in $\mathbb R^n$. 
The singleton  cone $C=\{0\}$ is regular by definition.

A fan is {\it regular} (``nonsingular''
in \cite{oda}, ``primitive'' in 
 \cite{bey77due}) if so are all its cones. 

\begin{lemma}
\label{lemma:regular}
Let $\Phi$ be a fan  in $\Rn.$
Then the regularity of $\Phi$  is decidable, once
the set $V_C$ of primitive generating vectors of each 
 cone $C\in \Phi$ is explicitly given. 
\end{lemma}

\begin{proof} Let  $C$ be a cone of $\Phi$ with
its
primitive generating vectors  $v_1,\ldots,v_s\in \mathbb Z^n$.
From Minkowski's classical convex body theorem,
\cite[Theorems 446-447]{harwri}, it follows  that
$C$ is regular iff
the half-open parellelepiped 
\begin{equation}
\label{equation:parallelepiped}
P_C=\{\lambda_1v_i+\dots+\lambda_sv_s \mid
0\leq \lambda_i<1,\,\,i=1,\ldots,s\}
\end{equation}
does not contain any integer point except 0.
Exhaustive search of such nonzero integer
point in $C$ will then settle the problem whether
$C$ is regular.
\end{proof}

\begin{lemma}
\label{lemma:smooth}
For any nonempty closed set $W\subseteq \Rn$ the following
conditions  are equivalent:
\begin{itemize}
\item[(a)] $W=\bigcup_iC_i$ for finitely many  rational 
simplicial cones
$C_i$  in $\Rn.$

\smallskip
\item[(b)] $W$ is the support of a regular fan  $\nabla$  in  $\Rn.$
\end{itemize}
Further, the map  $\bigcup_iC_i\mapsto \nabla$  is 
computable, once the  coordinates
of the primitive generating vectors
of the faces of each  $C_i$ are explicitly given. 
\end{lemma}

\begin{proof}
For the nontrivial direction,  we first compute 
 a fan  $\Phi$ with   $|\Phi|=W.$  This is the homogenous
 version of a well known result  in piecewise linear topology stating that
 every finite union $U = \bigcup T_i\subseteq \Rn$ of simplexes
 $T_i$  in $\Rn$
 has a triangulation  $\mathcal T$, \cite[\S 2, p.32]{sta}. 
If the  vertices of each 
 $T_i$  are rational,  then so are the
 vertices all simplexes in  $\mathcal T$, and
$\mathcal T$ is  computable.
 Similarly the fan
 $\Phi$ is   computable,
once the rational coordinates of the primitive
generating vectors of each cone $C_i$ are explicitly 
given. 
Next we construct the desired regular subdivision 
of   $\Phi$ into a regular fan  $\nabla$ with 
$|\Phi|=\nabla$, using 
 the ``desingularization'' procedure in   \cite[proof of Theorem VI.8.5, p.253]{ewa},
or \cite[Proposition 2.1]{bey77due}.  
Perusal of the proofs of these theorems,
or familiarity with the desingularization of singular fans,
shows that the sequence of stellar subdivisions
$\Delta_0=\Phi, \Delta_1,\Delta_2,\ldots,\Delta_{u-1}, \Delta_u=\nabla$
in the desingularization procedure  is 
computable. At each  step
$\Delta_t\mapsto \Delta_{t+1},$ assuming
the fan  $\Delta_t$ is not regular (a condition that can be decided, by
{Lemma
\ref{lemma:regular}}) we pick an integer point $x$ of
the half-open parellelepiped $P_C$
of some  cone $C\in \Delta_t$ witnessing
the non-regularity of $C$, 
(notation of \eqref{equation:parallelepiped}). 
Again,   $x$ is given by Minkowski convex body theorem. Then
we let  $\Delta_{t+1}$ be the new fan obtained
by ``starring''    $ \Delta_t$ at $x$. Such  
stellar subdivision operation
is purely combinatorial, \cite[III, Definition 2.1]{ewa}.
 The procedure stops when a
regular subdivision $\nabla$ of $\Phi$ is obtained.
In conclusion,  the
map $\Phi\mapsto \nabla$  is effectively computable, whence so is the map
$W\mapsto  \nabla$.
\end{proof}

\noindent{\it $\mathsf B$-maps,\footnote{As an effect of
 {Lemma  \ref{lemma:extension}}, real-valued
$\mathsf B$-maps are also  known as 
``piecewise  homogeneous linear
functions with integer coefficients''
 \cite{and, glamad,gla-book}, 
``$\ell$-maps with integer coefficients'' \cite{bak}, 
``integral $\ell$-maps'' \cite{bey75, bey77due},
 ``radiant functions`` \cite{mun-tams},  
 ``integral pwhl-maps'',
etc.  We have taken the liberty of increasing the terminological entropy
of  these functions, introducing  the two-syllabled 
neologism ``$\mathsf B$-map'' as the only new notational 
stipulation in this paper.}
$\mathsf B$-homeomorphism.}
Given a fan $\Delta$ in $\Rn$,  by a
$\mathsf B$-{\it map}   $f \colon |\Delta| \to \Rm$ we mean 
  a piecewise homogeneous
linear map such that each linear piece of $f$ has integer coefficients.
More precisely, $f$ is continuous and
 there are  linear homogeneous functions
$l_1,\ldots,l_k\colon |\Delta| \to \Rm$ with integer coefficients such that
for each  $x\in  |\Delta|$ there is $j\in\{1,\ldots,k\}$
with $f(x)=l_j(x).$
In particular, given  fans  $\Delta$ in $\Rn$ and $\nabla$ in $\Rm$, a
{\it $\mathsf B$-homeomorphism}\footnote{Also known
as integral piecewise homogeneous linear homeomorphism,
or $\ell$-equivalence.} 
of $|\Delta|$ onto $|\nabla|$ is
an invertible $\mathsf B$-map  $h$ of $|\Delta|$ onto $|\nabla|$  such that
also $h^{-1}$ is a $\mathsf B$-map.  

\smallskip
For $n=1,2,\dots,$   the free $n$-generator $\ell$-group  
 $\mathcal A_n$ coincides with
the $\ellg$  of    $\mathsf B$-maps 
   $f\colon \Rn\to \mathbb R$,
equipped with the
pointwise operations of the
$\ellg$  $(\mathbb R,0,+,-,\vee,\wedge).$
See \cite[Theorem 6.3]{and} for a proof.

\begin{lemma}[Extension]
\label{lemma:extension}
Let $\Delta$ be a  fan   in $\Rn$. Then every  
$\mathsf B$-map $f\colon |\Delta|\to \mathbb R$ can be extended to a
  $\mathsf B$-map   $g \colon \Rn\to \mathbb R.$ 
In symbols, 
\begin{equation}
\label{equation:extension}
f\in \mathcal A_n\restrict |\Delta|=
\{g\restrict |\Delta|  \mid g\in \mathcal A_n\},
\end{equation}
where   $\restrict$ denotes restriction.
\end{lemma}
\begin{proof} 
Use  \cite[Corollary 1 to Theorem 3.1]{bey75},
and note that the general extension argument given there
 still holds  for complexes
of  {\it rational} cones, provided we replace
Baker's characterization of finitely generated
projective vector lattices, \cite[Theorem 5.1]{bak},
by Beynon's characterization of 
finitely generated
projective $\ellg$s, \cite[Theorem 3.1]{bey77due}.  
\end{proof}

\medskip
\noindent
{\it The maximal spectral space of $\mathcal A_n$.}
Suppose $G\not= \{0\}$ is a finitely generated $\ell$-group. 
The {\it  spectral  (hull-kernel)}  topology, \cite{bkw}, 
makes the set  of maximal
$\ell$-ideals of $G$ into a nonempty compact Hausdorff space
$\maxspec(G)$, called the
 {\it maximal spectral space}   of $G$.
   The
closed sets in  $\maxspec(G)$ 
have the form 
\begin{equation}
\label{equation:spectral}
\mathcal{Z}{(\mathfrak h)}=\bigcap_{g\in \mathfrak h} \{\mathfrak{m} \in 
\maxspec(G) \mid  g
\in \mathfrak{m} \},
\end{equation}
for  $\mathfrak h$ ranging over all    $\ell$-ideals of $G$.
 
  \begin{lemma}[Maximal spectral space]
\label{lemma:sphere}
$\maxspec(\mathcal A_n)$ is homeomorphic to
the $(n-1)$-sphere  $S^{n-1}$, in symbols, 
$$
\maxspec(\mathcal A_n) \cong S^{n-1}.
$$
 More generally, 
in the notation of \eqref{equation:extension}, for every
fan  $\Phi$ in $\Rn$, 
$$\maxspec(\mathcal A_n\restrict |\Phi|)\cong |\Phi|\cap S^{n-1}.$$  
\end{lemma}
\begin{proof} 
Although this result is well known to specialists,
we sketch the proof to help the reader. 
We will tacitly  use  \eqref{equation:extension} 
in  {Lemma \
 \ref{lemma:extension}}.
 Given a vector $v\in \Rn$, the {\it ray}\,\, $\mathbb R_{\geq 0}v$
 is the halfline  $\{\lambda v \in \Rn\mid 0\leq  \lambda \in \mathbb R\}$.
For every ray  $\rho\subseteq \Rn,$
the set of all functions in $\mathcal A_n$ vanishing
over $\rho$ is a maximal $\elli$ $\mathfrak m _\rho$ of  $\mathcal A_n$.
Conversely, for every maximal $\elli$  $\mathfrak m$ of
$\mathcal A_n$ the intersection of the zerosets  
\begin{equation}
\label{equation:zeroset}
Zf=f^{-1}(0)
\end{equation}  of
all functions  $f\in \mathfrak m$ is a ray  $\rho_\mathfrak m$
in $\Rn.$  (As a matter of fact,
if  $Zf=\{0\}$ then $\mathfrak m=\mathcal A_n$,
against the definition of $\mathfrak m$. If $Zf$ contains
two or more rays  $\rho,\rho'$, then $\mathfrak m$ is 
strictly contained in $\mathfrak m_\rho$, which is impossible.)
The  maps  $\rho\mapsto \mathfrak m_\rho$ and 
$\mathfrak m \mapsto \rho_\mathfrak m$ are the inverse of
each other. 
Evidently,  rays are in one-one correspondence with their
intersections with $S^{n-1}$.
The definition 
\eqref{equation:spectral}
of the maximal spectral topology
is to ensure that the correspondence
$\mathfrak m\mapsto \rho_\mathfrak m\cap S^{n-1}$ is a homeomorphism.

The general case with $|\Phi|$ in place of $\Rn$ is proved in a similar way.
\end{proof}

\section{Applications of Baker-Beynon duality}
Building on previous work by Baker \cite{bak} on vector lattices,
Beynon  
constructed a duality $\mathcal D$ between
 pointset unions of finite sets of 
 rational  simplicial cones  with their  $\mathsf B$-maps, and
finitely presented $\ell$-groups with their $\ell$-homomorphisms,
\cite[Corollaries 2-3]{bey77due}.
In our fan-theoretic framework, in view of
 {Lemma \ref{lemma:smooth}}
 the functor $\mathcal D$  has the following
equivalent definition:

\medskip

\begin{description}
\item[{\it Objects}]
  For any   fan 
 $\Delta$ in  $\Rn$, 
\begin{equation}
\label{equation:functor-objects}
\mathcal D(|\Delta|)=\{f\restrict 
|\Delta| \mid f\in \mathcal A_n\}=\mathcal A_n\restrict |\Delta|=\{
\mbox{the $\ellg$ of all    $\mathsf B$-maps on $ |\Delta|$}\}.
\end{equation}

\medskip
\item[{\it Arrows}]
Given  fans  $\Delta$ in $\Rn$ and $\nabla$ in $\Rm$, 
and a $\mathsf B$-map  $b \colon |\Delta| \to |\nabla|$, 
\begin{equation}
\label{equation:functor-arrows}
\mathcal D(b) \colon f\in \mathcal D(|\nabla|)\mapsto f \circ b  \in 
\mathcal D(|\Delta|),  \mbox{ for short, } \mathcal D(b)=-\circ b
\colon  \mathcal D(|\nabla|)\to  \mathcal D(|\Delta|),
\end{equation}
where ``$\circ$'' denotes composition.
\end{description}

\bigskip
\noindent
In  particular,  $\mathcal D(\Rn)=\mathcal A_n$, and
$\mathcal D(\{0\})$ is the trivial  $\ellg$  $\{0\}$. 

\bigskip
Using {Lemma \ref{lemma:extension}}, Beynon's main
result     \cite[Corollary 2]{bey77due}  can be equivalently  stated as follows:

\begin{theorem}[Duality]
\label{theorem:duality}
$\mathcal D$  determines  a duality  between
supports of  fans with their  $\mathsf B$-maps, and
finitely presented $\ell$-groups with their $\ell$-homomorphisms.
Equivalently, by  {  Lemma \ref{lemma:smooth}}, 
$\mathcal D$  gives a duality  between
supports of {\em regular}  fans with their  $\mathsf B$-maps, and
finitely presented $\ell$-groups with their $\ell$-homomorphisms.
\end{theorem}


\begin{lemma}
\label{lemma:gen-range}
Given elements $f_1,\ldots  f_n\in \mathcal A_m$
let  $f\colon \Rm\to \Rn$ be defined by
$
f(x)=(f_1(x),\ldots,f_n(x)), \mbox{ for all } x=(x_1,\ldots,x_m)\in \Rm.
$
Let  
\begin{equation}
\label{equation:gen}
\gen(f)=\gen(f_1,\ldots,f_n)
\end{equation}
 be the
$\ell$-subgroup of $\mathcal A_m$ generated by $f_1,\ldots,f_n.$ Then
$\gen(f)\cong \mathcal D(\range(f)).$
\end{lemma}

\begin{proof}
Let    $R=\range(f)$.
The map
$l \in \mathcal D(R)\mapsto l\circ f \in \mathcal A_m$
 yields an $\ell$-ho\-m\-o\-mor\-ph\-ism
of $\mathcal D(R)$ into $\gen(f).$
The map is one-one because if  $0\not=g \in \mathcal D(R)$
(say, $g(x)\not= 0$ for some $x\in R$) 
then $0\not=g \circ f(z)$, where
$z$ is  such that $f(z)=x$.
The map  is onto $\gen(t)$ because, by
{  Lemma \ref{lemma:extension}},   every   $h \in \gen(f)$
has the form $g= s \circ f$ for some  $s\in \mathcal A_n$.
\end{proof}

\begin{lemma}
[Fans over zerosets]
\label{lemma:zeroset=support}
For every $f\in \mathcal A_n$  there is
a  regular  fan $\Phi$ such that
$|\Phi|=Zf=f^{-1}(0).$  
 Once $f$ is specified as $f=\hat t$
for some $\ell$-group term  $t(X_1,\ldots,X_n)$,  the map
$t\mapsto \Phi$ is computable.
\end{lemma}

\begin{proof}  We first  construct a regular fan 
$\Lambda_f$ in $\Rn$  that {\it linearizes} $f$,
in the sense that $f$ is linear over each cone of $\Lambda_f.$
To this purpose, we list the linear pieces  $l_1,\dots,l_u$
of $f$.  Once $f$ is specified as  $f=\hat t$ for some  
$n$-varlable 
$\ellt$  $t$, the $l_i$ are obtained effectively---by
induction on the number of operation symbols in $t$.
Next, for each permutation $\pi$ of the index set 
$\{1,\dots,u\}$ we let the cone  $C_\pi\subseteq \Rn$ be defined by
$$
C_\pi = \{x\in \Rn\mid l_{\pi(1)}(x)\leq\cdots\leq l_{\pi(u)}(x)\}.
$$
For any such $\pi$, $f$ is linear over  $C_\pi.$ We next compute
the sequence  of regular fans $\Delta_1,\Delta_2,\dots$ 
defined in \cite{mun-tams}, 
until    $\Delta_j$  has the property that
every  $C_\pi$ is a union of cones of  $\Delta_j.$  The existence
of  such $j$ is ensured by  
 \cite[Lemma 3.7]{mun-tams}. The fan $\Delta_j$  provides
 the desired linearization $\Lambda_f$.  
We finally  let
$\Phi=\{C\in\Lambda_f \mid \mbox{$f$ constantly vanishes over $C$}\}$.
By construction, the map $\Lambda_f\mapsto \Phi$ is computable, whence
so is the map $t\mapsto \Phi$.
\end{proof}

\begin{lemma}
[Fans over $\mathsf B$-images]
\label{lemma:range=support}  Let $b\colon \Rm\to \Rn$
be a $\mathsf B$-map. Then
there is
a  regular  fan $\Theta$ such that
$|\Theta|=\range(b)$. Once  $b$ is
specified as \,\, $b=\hat t=(\hat t_1,\ldots,\hat t_n)$\,\, for
  $\ell$-group terms  \,\,$t_i(X_1,\ldots,X_m)$, the map
$(t_1,\dots,t_n)\mapsto \Theta$ is computable.
\end{lemma}


\begin{proof} 
Using
 \cite[Lemma 3.7]{mun-tams}   
 as in the proof of 
{Lemma \ref{lemma:zeroset=support}}, we first
compute a regular fan $\Lambda_b$ in $\Rm$ that linearizes
$b$. 
 The image  $b(C_i)\subseteq \Rn$ of each cone $C_i\in \Lambda_b$
is  the positive hull of uniquely given
integer vectors  $v_{i1},\ldots,v_{is_{i}}\in \Rn,$  which 
are effectively computable from $C_i$ and $t$.
 The effective procedure
of {Lemma \ref{lemma:smooth}} now provides the required
regular fan $\Theta$ in $\Rn$ with 
$|\Theta|=\range(b)=\bigcup\{C_i\mid C_i\in
 \Lambda_b \}$.
So the map $t\mapsto \Theta$ is computable.
\end{proof}


\medskip
For later use we state
Beynon's characterization of finitely
generated projective $\ellg$s \cite[Theorem 3.1]{bey77due}
in the following expanded form:

\begin{theorem} [Characterization of projectives]
\label{theorem:projective} 
For any $n$-generator $\ell$-group  $G$ the following conditions are
equivalent:
\begin{itemize}
\item[(a)] $G$ is  projective.
\item[(b)]  $G$ is $\ell$-isomorphic to a finitely
presented $\ell$-group  $\mathcal A_n/\mathfrak p$,
where $\mathfrak p$  is a principal $\ell$-ideal 
of $\mathcal A_n$,
say, $\mathfrak p=\langle p \rangle =$ the $\ell$-ideal generated
by   $p\in \mathcal A_n$.
\item[(c)]  
$G$ is $\ell$-isomorphic to the   $\ell$-group\,  $\phi(\mathcal A_n)$, where
$\phi\colon \mathcal A_n\to \mathcal A_n$ is an  
idempotent endomorphism.
\item[(d)]   
$G$ is $\ell$-isomorphic to
$\mathcal D(O)$, where $O$ is the zeroset $Zf=f^{-1}(0)$
of some  $f\in \mathcal A_n$.

\end{itemize}
\end{theorem}
 
\begin{proof}
The equivalence 
(a)$\Leftrightarrow$(b) was proved by 
Beynon in  \cite[Theorem 3.1]{bey77due}.

\smallskip

(a)$\Leftrightarrow$(c) is a special case of a 
folklore result in universal 
algebra.

\smallskip
(b)$\Leftrightarrow$(d) follows from the canonical $\ell$-isomorphism
$\mathcal A_n/\langle p \rangle \cong  \mathcal A_n\restrict Zp$, where
$Zp$ denotes the zeroset of $p$. 
As a matter of fact, for  a function
$q \in \mathcal A_n$ to  belong to the principal $\elli$
 $\langle p \rangle\subseteq \mathcal A_n$ it is
necessary and sufficient that $Zq$ contains $Zp.$
Thus, functions  $r,r'$ satisfy the condition  $|r-r'|\in \langle p \rangle$
iff their restrictions to $Zp$ coincide.
(Caution: this argument
may fail  if  $\langle p \rangle$
is replaced by a nonprincipal ideal.)
\end{proof}

\begin{remark}
\label{remark:1}
  If in  Theorem \ref{theorem:main}
we assume $n>m$,  
the answer to the  problem  
 is automatically negative.  As a matter of fact,  since
 $\hat t$ is continuous, 
{ Lemma \ref{lemma:sphere} } 
yields
 $\dim(\maxspec(\mathcal D(\range(\hat t))))
 =\dim(\range(\hat t)\cap S^{n-1}) \leq m-1 < n-1
=\dim(\maxspec(\mathcal A_n)).$  
By {Lemma
 \ref{lemma:gen-range}}, $\dim(\maxspec(\mathcal D(\range(\hat t))))
 =\dim(\maxspec(\gen( \hat t)))\not=\dim(\maxspec(\mathcal A_n))$,
 whence $\maxspec(\gen( \hat t))$ is not homeomorphic to 
 $\maxspec(\mathcal A_n)$, and a fortiori 
$\gen( \hat t)$ is not  $\ell$-isomorphic to $\mathcal A_n$.
 Thus the nontrivial part of 
  { Theorem \ref{theorem:main}}  is when $n\leq m.$  
\end{remark}

\begin{remark}
\label{remark:2}
If in  the statement of Theorem \ref{theorem:main-due}
we assume  $n<l$,  the answer is automatically negative,
because  $\mathcal A_l$ does not have a generating set with
$n$ elements only. Otherwise (absurdum hypothesis),
by {Lemma \ref{lemma:sphere} } we have
$l-1=\dim(S^{l-1})=\dim(\maxspec(\mathcal A_l))\leq
 \dim(\maxspec(\mathcal A_n))=\dim(S^{n-1})=n-1,$
 which is impossible. So the nontrivial part of
 Theorem \ref{theorem:main-due} is when $n\geq l.$ 
\end{remark}

\noindent
{\it The abstract simplicial complex of a fan.}
Following \cite[p.244]{bey77due},  from any  (always
rational and simplicial)  fan  $\Phi$
we obtain the abstract simplicial complex  $\overline{\Phi}$
via the following construction:

For each cone $C\in \Phi$ let $\partial C$ denote the
primitive generating vectors of $C$.  Then the set $|\Phi|$
of points
of $\overline{\Phi}$ is the set  of primitive generating vectors
of all cones of $\Phi.$
More generally, an element of $\overline{\Phi}$ has the form
$\partial C$, where  $C$ ranges over the totality of cones of $\Phi$.  
 In symbols,
\begin{equation}
\label{equation:asc}
\overline{\Phi}= \{\partial C\mid C\in \Phi\},\,\,\,\,\,\, |\overline{\Phi}|
=\bigcup\{\partial C\mid C\in \Phi\}.
\end{equation}
 The   order  $\leq$  in  $\overline{\Phi}$  
is given by  the inclusion order between the sets  $\partial C$, for
$C\in\Phi.$ Since each primitive generating vector
of a cone in $\Phi$  has integer coordinates,
the  map $\Phi\mapsto \overline\Phi$ is  computable,
once each  cone of  $\Phi$  is specified via its primitive
generating vectors.  
Conversely, with  the notation of \eqref{equation:pos}, 
 $\Phi=\{\pos(K)\mid K\in \Phi\}$.

Two abstract simplicial complexes $\mathcal C,\mathcal C'$ are {\it isomorphic}
if there is a one-one map  $\beta$  of $|\mathcal C|$ onto $|\mathcal C'|$
such that 
for any two elements  
$a,b\in \mathcal C,$
$a\leq b$  iff 
$\beta(a)\leq\beta(b)$.  Isomorphism is decidable, through exhaustive search
over all one-one maps of $|\mathcal C|$ onto $|\mathcal C'|$.

\begin{lemma}
\label{lemma:combinatorial}
Given  fans  $\Phi$ in $\Rn$ and $\Psi$ in $\mathbb R^d$,
we have the following equivalent conditions:
\begin{itemize}
\item[(i)] $\mathcal D(|\Phi|)\cong\mathcal D(|\Psi|)$.

\smallskip
\item[(ii)] There is a $\mathsf B$-homeomorphism  $h$   of
$|\Phi|$ onto  $|\Psi|$,  say,  $h=\hat t$ for some $d$-tuple
of $n$-variable $\ell$-group terms.

\smallskip 
\item[(iii)]  $\Phi$ and   $\Psi$ have regular
subdivisions
$\Delta$ and   $\nabla$  with an isomorphism
$\beta$ of   $\overline{\Delta}$ onto
$\overline{\nabla}$.
\end{itemize}

\end{lemma}

\begin{proof}  (i)$\Leftrightarrow$(ii) is 
an immediate consequence of {Theorem \ref{theorem:duality}}.

(iii)$\Rightarrow$(ii) follows from  \cite[Corollary 3]{bey77due}.
Precisely the assumed regularity of 
 $\Phi$ and   $\Psi$ ensures that the  existing  
 isomorphism  $\beta$ of  $\overline{\Delta}$ onto
$\overline{\nabla}$ determines
a piecewise linear homeomorphism $h_\beta$ of
 $|\Delta|$ onto $|\nabla|$ which linearly maps
 each cone of $C\in \Delta$ onto 
 a cone  $h(C)\in \nabla$. 
 Evidently,  $h_\beta$ is a $\mathsf B$-homeomorphism
 of $|\Delta|=|\Phi|$  onto  $|\nabla|=|\Psi|$.
Moreover,  every linear piece $l$
 of  $h_\beta$   can be expressed
 by a suitable  integer $d \times n$ matrix  $M_l$
 with  $l(x)=M_lx$ for all $x\in \Rn$. 
 Arguing as in \cite[Proposition 2.3]{mun-tams},
  we can write 
 $h_\beta=\hat t=(\hat t_1,\dots,\hat t_d)$ for suitable $\ellt$s   
 $t_i$. Specifically, each  $\hat t_i$ is expressible as a linear
 integer combination of
 $\Delta$-linear
 support functions (\cite[Definition, p.66]{oda}), called
  ``Schauder hats'' in \cite{mun-tams}.
   By direct inspection of \cite[\S 2]{mun-tams},     
  each $t_i$ is computable from the input data $\beta,\Delta,\nabla$. 

\smallskip
(ii)$\Rightarrow$(iii)  Let
$h$ be a  $\mathsf B$-map  
of $|\Phi|$  onto  $|\Psi|$, specified
as  $h=\hat t$  for a  $d$-tuple of
$\ellt$s
$t=(t_1,\dots,t_d)$, each $t_i$ in the variables  $X_1\dots,X_n.$
In order to
preliminarily check that  $\hat t$ is indeed a
$\mathsf B$-homeomorphism, 
we  compute a  regular subdivision $\Delta_h$ of
 $\Phi$    such that 
$h$ is linear on each cone  of $\Delta_h$, and
$h(C)$ is contained in some cone of $\Psi$, for 
each $C\in \Delta_h.$ 
For the effective computability of  $\Delta_h$  
we may again refer to the uniform linearization procedure of  
\cite[Lemma 3.7]{mun-tams}.
 Next, using { Lemma \ref{lemma:regular}}
   we check whether  the image 
  $$h(\Delta_h)=\{h(C)\mid C\in \Delta_h \}$$ 
is a {\it regular} subdivision of $\Psi$. This is necessary
and sufficient for $h$ to be a $\mathsf B$-homeomorphism
of $|\Phi|$  onto  $|\Psi|$.
The inverse map
$h^{-1}$ is linear over each cone  $D$  of
$h(\Delta_h)$, and maps $D$ onto a regular cone of
$\Delta_h$.  Upon setting 
 $\Delta=\Delta_h$ and  $\nabla=h(\Delta_h)$, from
the primitive generating vectors of the faces of
 $\Delta$ and  $\nabla$ we can easily transform
  $h$ into the desired  
isomorphism $\beta$ of  $|\overline \Delta|$ onto $|\overline \nabla|$.
The map $t\mapsto \beta$ is computable.
\end{proof}

\begin{theorem}[Undecidability]
\label{theorem:glamad}
 
The following problem is undecidable:

\smallskip
\noindent
${\mathsf{INSTANCE}:}$  An 
$\ell$-group  term  $u(X_1,\ldots,X_k)$ and an integer $l>0$. 
 
\smallskip
\noindent
${\mathsf{QUESTION}:}$
Letting $\langle \hat u \rangle$ be the $\ell$-ideal  of
$\mathcal A_k$ generated by $\hat u$, is  
  the quotient $\ell$-group  $\mathcal A_k/\langle \hat u \rangle$
$\ell$-isomorphic to   $\mathcal A_l$ ?
\end{theorem}
\begin{proof}
This is proved in 
 \cite[Theorem D]{glamad},
using
{Lemma \ref{lemma:sphere}},  Beynon's duality
{ (Theorem \ref{theorem:duality}),}
and  S. P. Novikov's  theorem,
\cite{sht},  \cite[\S 3]{chelek},
 on the  unrecognizability of the sphere
$S^n$, for $n\geq 5$. 
\end{proof}

 \section{Proof of Theorem \ref{theorem:main}}
 
 \begin{proof}
Let  $\hat{t}$ denote the $\mathsf B$-map 
$(\hat t_1,\ldots,\hat t_n) \colon \Rm \to \Rn.$

\smallskip

\medskip
\noindent
{\it Claim 1.}  $ \mathcal A_n$  is not $\ell$-isomorphic to the $\ellg$
$\gen(\hat t)$ generated by $t_1,\dots,t_n$   \,\, iff \,\, there is a  point $x\in \Rn\setminus \range(\hat t).$
 
\smallskip
If  $\Rn\setminus \range(\hat t) = \emptyset$  then by definition of $\mathcal D,$ 
$\mathcal A_n=\mathcal D(\range(\hat t))\cong
\gen(\hat t)$, by 
{ Lemma \ref{lemma:gen-range}}. 
Conversely, assume  $\Rn\setminus \range(\hat t) \not= \emptyset$,
say  $x\in \Rn\setminus \range(\hat t)$. 
As in the proof of { Lemma
\ref{lemma:zeroset=support}}, we construct a regular fan $\Lambda$
such that $\hat t$ is linear over each cone $C_i\in \Lambda.$ 
Each  image  $\hat t(C_i)$  
coincides with the positive hull of uniquely given
integer vectors that can be computed effectively. 
Thus we may assume
  $x\in S^{n-1}\subseteq \Rn$  without loss of generality.
By {Lemmas \ref{lemma:gen-range} and \ref{lemma:sphere}},
 $\maxspec(\mathcal D(\range(\hat t)))\cong
\maxspec(\gen( \hat t))\cong S^{n-1}\setminus X$ for
some set $X\subseteq S^{n-1}$   containing $x$;
further,  $\maxspec(\mathcal A_n)\cong S^{n-1}.$
Observe that   $S^{n-1}$ is not homeomorphic to 
$S^{n-1}\setminus X$:  as a matter of fact,   
stereographic projection from $x$ yields a homeomorphism
of $S^{n-1}\setminus \{x\}$ onto $\mathbb R^{n-1}$.
Since  $X \subseteq S^{n-1}\setminus \{x\}$ then 
some homeomorphic copy $X'$ of $ X$  is embedded 
into  $\mathbb R^{n-1}$. But, as is well known,  
$S^{n-1}$ is not embeddable into $\mathbb R^{n-1}$.
We conclude that $S^{n-1}\setminus X$
 is not homeomorphic to $S^{n-1}$.
(For details see, e.g.,  \cite[Example 7.2(6),
p.180]{ams}).  It follows that 
$\maxspec(\mathcal A_n)$ is not homeomorphic
to $ \maxspec(\gen( \hat t))$.
Therefore,  $\mathcal A_n$
 is not $\ell$-isomorphic
to $\gen(\hat t)$.  Claim 1 is settled.

\smallskip 
\noindent
{\it Claim 2.}  The two $\ellg$s
$\gen (\hat t)$  and  $ \mathcal A_n$ are $\ell$-isomorphic
 iff there are
regular fans  $\nabla$ and $\Delta$  with their supports
coinciding with 
$\range(\hat t)$ and  $\Rn$ respectively, such that their
associated abstract simplicial complexes
 $\overline \Delta, \overline \nabla$  are  
isomorphic.


$(\Rightarrow)$
From  $\gen(\hat t)  \cong \mathcal A_n$ and
 $\gen(\hat t)\cong \mathcal D(\range(\hat t))$
 {(Lemma \ref{lemma:gen-range})}, an
 application of {Theorem \ref{theorem:duality}}
yields  a $\mathsf B$-homeomorphism
of $\range(\hat t)$ onto $\Rn.$
The proof of  {Lemma \ref{lemma:combinatorial}(ii)$\Rightarrow$(iii)} 
shows how to compute
regular fans  $\Delta, \nabla$   
 with    $|\Delta|=\range(\hat t)$ and $|\nabla|=\Rn$
such that 
 $\overline \Delta$ is isomorphic to $ \overline \nabla$.

$(\Leftarrow)$
Since  $\overline \Delta, \overline \nabla$ are  
isomorphic, by  another application of 
{Lemma  \ref{lemma:combinatorial}(iii)$\Rightarrow$(ii)}  we compute 
a  $\mathsf B$-homeomorphism of their respective supports
$\range(\hat t)$ and  $\Rn$.
Then by  {Theorem \ref{theorem:duality} 
and Lemma \ref{lemma:smooth} }  we can write
$
\gen(\hat t)\cong \mathcal D(\range(\hat t))\cong \mathcal D(\Rn)
\cong \mathcal A_n.$

\smallskip
Having thus settled Claim 2, we next  let  the 
two  Turing machines, $\mathcal M$
and $\mathcal N$  run in parallel as follows:

%
%
 
\smallskip
 
\noindent
 {\it Machine $\mathcal M$}
 preliminarily constructs a regular fan  $\Delta$ in $\Rn$
 with $|\Delta|=\range(\hat t)$. In the light of
 {Lemma \ref{lemma:regular}} the construction is effective:
since  each $\mathsf B$-map  $\hat t_j$ is   computable,
  the homogeneous variant of the
triangulation process of  \cite[\S 2, p.32]{sta}
makes   $\range( \hat t)$  into 
a finite union of rational simplicial cones 
 in $\Rm$ whose primitive generating vectors are computable. 
 Next, $\mathcal M$ enumerates in some 
prescribed lexicographic order all rational points
in $\Rn$,  looking for some {\it rational} point $r\in \Rn$ that does not belong
to   $\range(\hat t)$. The existence of  such $r$ is given
by a routine refinement of Claim 1, upon noting that by hypothesis,  
 the range of $\hat t$ is a closed subset of $\Rn$ not coinciding
 with $\Rn.$
%
%
%
%
%
%
 Checking whether $r$ belongs to $|\Delta|$ is an effective
 operation: For each (maximal) cone  $C$ of $\Delta$,
  the problem whether $r$
lies in $C$  amounts
to solving a finite number of explicitly given linear inequalities with rational coefficients.
By Claim 1, the existence of such $r$ is a necessary and sufficient
condition for the  non-isomorphism of $\gen(\hat t)$
 and $\mathcal A_n.$

\medskip
\noindent
 {\it Machine $\mathcal N$}
  builds the sequence of all possible pairs of fans  $(\Phi_i,\Psi_i)$ 
 with   $|\Phi_i|=\Rn$  and   $|\Psi_i|=\range(\hat t),$ 
  as in the proof of
 {Lemma \ref{lemma:smooth}}, following
 some  prescribed lexicographic order. 
 Since regularity is decidable  (by {Lemma \ref{lemma:regular}}),
 $\mathcal N$ simultaneously
  generates the subsequence $(\nabla_i,\Delta_i)$ of all 
 pairs of  regular fans, with their corresponding  abstract simplicial 
 complexes
  $(\overline \nabla_i, \overline \Delta_i)$. At stage $j$,
  $\mathcal N$ checks whether  there is an
 isomorphism between 
   $\overline \nabla_j$ and $\overline \Delta_j$.
     This is carried on by exhaustive search
 among all
 possible one-one maps of the set of
 primitive generating vectors of cones of  $\nabla_j$
 onto the set of
 primitive generating vectors of cones of  $\Delta_j$.
By Claim 2, the existence of an isomorphism
 between
   $\overline \nabla_j$ and $\overline \Delta_j$ for some $j$, 
 is a necessary and sufficient  condition 
for the $\ellg$s\,\,
$\gen(\hat t(X_1,\ldots,X_n))$\,\,  and \,$\mathcal A_n$\,
to be $\ell$-isomorphic.

\smallskip
Precisely one of the competing machines
$\mathcal M$ and $\mathcal N$ will stop after
a finite number of steps. 
In this way we get the desired  decision 
procedure for the problem of
{ Theorem \ref{theorem:main}}, and complete its proof.
\end{proof}

 \section{Proof of Theorem \ref{theorem:main-due}}
\begin{proof} 
By way of contradiction, assume the decidability of the problem
\begin{equation}
\label{equation:nml}
\gen(\hat t_1(X_1,\ldots,X_m),\ldots,\hat t_n(X_1,\ldots,X_m))
\,\,\mbox{?`} \hspace{-0.9mm}\cong ?\,\,\mathcal A_l
\end{equation}
of { Theorem \ref{theorem:main-due}}.
A contradiction will be obtained
by defining a (Turing) reduction to Problem  \eqref{equation:nml} of
the  undecidable problem
\begin{equation}
\label{equation:glamad}
\mathcal A_k/\langle \hat u(X_1,\ldots,X_k) \rangle 
\,\,\mbox{?`}\hspace{-0.9mm}\cong\, ?\,\,  \mathcal A_l
\end{equation}
of  { Theorem \ref{theorem:glamad}}.
To this purpose we first  prove the following parametrization result:
 
 \medskip
 \noindent
 {\it Claim:}
 The zeroset $Z\hat u$ 
 coincides with the support of 
 some regular fan $\Delta$,
 and is    
$\mathsf B$-homeomorphic to the
range  of a  $\mathsf B$-map 
$h \colon \mathbb R^k \to \mathbb R^k$
whose range  
coincides with the support of some regular fan $\nabla$.

The first statement follows from {Lemma \ref{lemma:zeroset=support}}. 
Next, for the construction of the  $\mathsf B$-map 
$h$, { Theorem \ref{theorem:projective}}
shows that the $k$-generator 
projective  $\ellg$  $\mathcal A_k/\langle \hat u \rangle 
\cong \mathcal A_k\restrict Z\hat u$
  is $\ell$-isomorphic to 
a retract  $H$   
of   $\mathcal A_k$.  Thus for some 
idempotent endomorphism  $h =(h_1,\ldots,h_t)\colon
\mathcal A_k\to \mathcal A_k$ we can write
$H=\gen(h_1,\ldots,h_t)\subseteq \mathcal A_k$. 
%
%
%
%
By  { Lemma \ref{lemma:gen-range},} 
from the isomorphisms 
$
\mathcal A_k\restrict Z\hat u\cong
\mathcal A_k/\langle \hat u \rangle\cong H\cong \gen(h)
\cong \mathcal D(\range(h))
$
we get an isomorphism $\phi\colon \mathcal A_k\restrict Z\hat u\cong
\mathcal D(\range(h))$.   Recalling the notation
 \eqref{equation:functor-arrows}, 
  { Theorem \ref{theorem:duality}} 
yields the $\mathsf B$-homeomorphism  $\mathcal D(\phi)$   
of $\range(h)$ 
onto $Z\hat u$.   
By {Lemma  \ref{lemma:range=support}}
there is a regular fan  $\nabla$ whose
support  coincides with
 $\range(h)$.
This settles our claim.

\medskip

To conclude the proof,
 let    Turing machine $\mathcal E$ enumerate  all $k$-tuples 
$s_1,s_2,\ldots$ of  $\ellt$s in $k$ variables  $X_1,\ldots,X_k$ and
check, for any such $s_t=(s_{t1},\ldots,s_{tk})$ whether
 the zeroset  $Z\hat u$ is $\mathsf B$-homeomorphic to 
$\range(\hat s_t)\subseteq \mathbb R^k$.
For each $t=1,2,\ldots$, this
 $\mathsf B$-homeomorphism problem is recursively enumerable.
Indeed, in view of 
{Lemma \ref{lemma:combinatorial}}, we may suppose that,
for each $t$, 
$\mathcal E$ also 
 enumerates all pairs of
regular fans  $\Phi_{ti}$  with support    $\range(\hat s_t)$
 and $\Psi_{ti}$ with support
  $Z\hat u$ until two fans are found whose 
  combinatorial counterparts  
  $\overline \Phi_{ti},\overline\Psi_{ti}$
are isomorphic---which can be decided by exhaustive search.
By our claim, for some $t^*$ and $i^*$,
 a $k$ tuple $s_{t^*}$ of $\ellt$s and a pair  $\Phi_{t^*i^*},\Psi_{t^*i^*}$ 
of regular fans will be found, together with an
isomorphism  $\eta_{t^*i^*}$  of the two abstract simplicial complexes
  $\overline \Phi_{t^*i^*},\overline\Psi_{t^*i^*}$.
  This ensures the termination of machine $\mathcal E.$

Once the desired $k$-tuple   
$s_{t^*}$ is found out  
(along with the triple
 $\eta_{t^*i^*}$,   $\overline \Phi_{t^*i^*},\overline\Psi_{t^*i^*}$
 certifying the $\mathsf B$-homeomorphism
 of $\range(\hat s_{t^*})$ and   $Z\hat u$),      
%
by  {Theorem \ref{theorem:duality} and
Lemma \ref{lemma:gen-range}  }
 we can write
$$
\mathcal A_{k}\supseteq
\gen(\hat s_{t^*})\cong \mathcal D(\range(\hat s_{t^*}))\cong \mathcal D(Z\hat u)
\cong \mathcal A_k/\langle \hat u \rangle.
$$
As a consequence,  the answer to question  \eqref{equation:glamad} 
  is equal to the answer to the question
$\gen(\hat s_{t^*})   \,\,\mbox{?`}\hspace{-1.4mm}\cong ?\,\,  \mathcal A_l.$
The latter is decidable by our
absurdum hypothesis.
This yields the decidability of question \eqref{equation:glamad}, in contradiction 
with  {Theorem  \ref{theorem:glamad}}.
\end{proof}

\section{Proof of Corollary \ref{corollary:main-tre}, 
Concluding Remarks and a Problem}
\label{section:final}
We first prove that every
finitely generated free $\ell$-group $G$ is {\it hopfian},
in the sense that every surjective $\ell$-endomorphism
of $G$ is injective. 
This is the
$\ellg$-theoretic counterpart of a well known
basic property 
of many important classes of structures,
\cite{eva, hir, kos}.

\medskip
\begin{proposition}
\label{proposition:hopf}  Let $\eta$ be an
$\ell$-homomorphism of $\mathcal A_n$ onto
$\mathcal A_n$. Then $\eta$ is one-one.
\end{proposition}

\begin{proof}
The set $\{g_1,\ldots,g_n\}=\{\eta(\pi_1),\dots,
\eta(\pi_n)\}$ generates $\mathcal A_n$.
  $\eta$ is
 the unique $\ell$-ho\-m\-o\-mor\-ph\-ism of 
 $ \mathcal A_n$ onto $\mathcal A_n$ 
 extending the map $\pi_i\mapsto g_i,
 \,\,(i=1,\ldots,n).$
For each  $l\in \mathcal A_n$,
 $\eta(l)=l\circ (g_1,\ldots,g_n)=l\circ g$. 
If (absurdum hypothesis)
 $\eta$ is not one-one, 
 there is  $0\not= f\in \mathcal A_n$ such that $0= \eta(f)=f\circ g.$ 
Pick $x\in \Rn$ such that  $f(x)\not=0$.  Since $f$ is a
$\mathsf B$-map  we may safely assume
 $x\in S^{n-1}$.  Since $f$ 
constantly vanishes over
$\range(g)$ then   $x\notin \range(g)$.  By
{Lemmas   \ref{lemma:sphere} and \ref{lemma:gen-range}},
we have homeomorphisms $\maxspec(\mathcal A_n)\cong S^{n-1}$ 
along with  
$$
\maxspec(\gen(g))\cong\maxspec(\mathcal D(\range(g)))
\cong \range(g)\cong S^{n-1}\setminus X 
$$
for some set $X\subseteq \Rn$ containing  $x$.
We have already noted  (\cite[p.180]{ams}) that
$S^{n-1}$ is not homeomorphic to  $S^{n-1}\setminus X $.
As a consequence, the $\ellg$
 $\gen(g)$ is not $\ell$-isomorphic to $\mathcal A_n$,
whence a fortiori $\{g_1,\ldots,g_n\}$ is not a generating set
of $\mathcal A_n$, a contradiction.
 \end{proof}

\medskip

\noindent{\it Conclusion of the proof of Corollary \ref{corollary:main-tre}.}

\smallskip
\noindent{\it Claim.} The following conditions are equivalent:

\begin{itemize}
\item[(i)] $\{\hat t_1,\ldots,\hat t_n\}$ is a free generating set of
$\gen(\hat t)$;
\item[(ii)]  $\gen(\hat t)$ is $\ell$-isomorphic to $\mathcal A_n$.
\end{itemize}

\smallskip

(i)$\Rightarrow$(ii)  As is well known, up to $\ell$-isomorphism
there is a unique free $n$-generator $\ell$-group.

(ii)$\Rightarrow$(i)  By hypothesis,  $\gen(\hat t)$ has a
free generating set 
$\{h_1,\ldots,h_n\}$. The map
$h_i\mapsto \hat t_i,
 \,\,(i=1,\ldots,n)$ uniquely extends to an
 $\ell$-endomorphism $\psi$ of $\gen(\hat t).$
 By definition of $\gen(\hat t)$,  $\psi$ is onto
$\gen(\hat t).$ Proposition \ref{proposition:hopf}
(with $\gen(\hat t)$ in place of $\mathcal A_n$, and
$\psi$ in place of  $\eta$) ensures that $\psi$ is
one-one, whence it is an $\ell$-automorphism of
$\gen(\hat t)$.  
As an $\ell$-isomorphic copy (via $\psi$) of the free
generating set $\{h_1,\ldots,h_n\}$,
the set  $\{\hat t_1,\ldots,\hat t_n\}$ itself is
free generating in $\gen(\hat t).$

\smallskip
Having thus settled our claim, 
the problem whether $\{\hat t_1,\ldots,\hat t_n\}$
is a free generating set of  $\gen(\hat t)$ has the same answer
as the problem whether  
 $\gen(\hat t)$ is $\ell$-isomorphic to $\mathcal A_n$. The latter
problem  is   decidable,   by  {Theorem \ref{theorem:main}}. 
$\qed$

\smallskip
$$
***
$$

\bigskip
\noindent
In this paper  every $n$-generator  
$\ells$  $G$  of a free $\ellg$  $\mathcal A_m$ has been 
coded 
by  $\ellt$s   $t_i(X_1,\ldots,X_m),\,\,\,(i=1,\dots,n)$ 
denoting  a generating set    $\{\hat t_1,\ldots,\hat t_n\}$
 of $G$. 
This coding  is no less expressive than the traditional one---where  $G$ is presented as a principal
quotient of a  free $\ellg$.
Further,  the $\ellg$ terms  $t_i$ provide a convenient method 
of presenting rational polyhedra as finite strings of symbols. 
As a matter of fact, letting the map
$\hat t \colon \Rm\to \Rn$  be defined by
 $\hat t=(\hat t_1,\ldots,\hat t_n) $, the
set   
$$P_{\hat t}=\range(\hat t)\cap \bd[-1,1]^n$$
 is the most general rational polyhedron in $\Rn$ with 
dimension  $\leq n-1,$\,\,(Theorem
\ref{theorem:duality}, Lemma
\ref{lemma:zeroset=support}).
 
Given  generators  $t_1,\dots,t_n$, the
recognition problem of the $\ellg$  $G$,
(resp.,  of
the dual  rational polyhedron  $P_{\hat t}$)
depends on three parameters:

\begin{itemize}
\item[(i)] The dimension   of the ambient space $\Rn$ of $\range(\hat t)$,
i.e., the number of terms $t_i;$
\item[(ii)]  The dimension $m$ of the domain  $\Rm$ of the parametrization
$\hat t,$  i.e., the number of variables of each $t_i;$
\item[(iii)] The dimension $l-1$ of the sphere to which
$\range(\hat t)$ is compared, or the number $l$
of free generators of the free $\ellg$  $\mathcal A_l$  to which $G$ is  compared.   
 \end{itemize}

\noindent
Depending on these  parameters we get 
nontrivial undecidability and  (surprisingly enough)
  decidability results for recognition and isomorphism problems
  of the rational polyhedra $P_{\hat t}$, 
involving the fine structure of
finitely generated projective $\ellg$s, their maximal
spectral spaces, and their associated regular fans.
Remarks \ref{remark:1}, \ref{remark:2},
as well as \ref{remark:final} below, 
discuss the interplay between  
$n,m,l$.

\begin{problem}
\label{problem:final}
Fix $n=1,2,\dots,$ and let $\mathcal P_n$ be the following
problem:

\smallskip
\noindent
${\mathsf{INSTANCE}:}$  
$\ell$-group  terms  $t_1,\ldots,t_{n+1}$   in
the same variables $X_1,\ldots,X_n$.  
 
\smallskip
\noindent
${\mathsf{QUESTION}:}$  
Is   $\mathcal A_n$\,\,
$\ell$-isomorphic to
  the $\ellg$  
generated by  $\hat t_1,\ldots,\hat t_{n+1}$?

\smallskip
\noindent For which $n$ is  $\mathcal P_n$ decidable?
\end{problem}

\smallskip
\begin{remark}
\label{remark:final}
{ Lemma  \ref {lemma:gen-range}}  yields  homeomorphisms
 $$\maxspec(\mathcal D(\range(\hat t)))
 \cong\maxspec(\gen(\hat t_1,\dots,\hat t_{n+1}))
 \cong \range(\hat t).$$
Thus by  {Theorem \ref{theorem:duality},} problem 
$\mathcal P_n$  calls for a mechanical method to decide 
whether or not  the maximal spectral space
 $\maxspec(\gen(\hat t))=
 \maxspec(\gen(\hat t_1,\dots,\hat t_{n+1}))$ 
 is homeomorphic to the sphere
  $S^{n-1},$ or else, a proof that such method does not exist. 
{By Lemma \ref{lemma:sphere}} we are left with the
following   equivalent  geometric counterpart of
problem  $\mathcal P_n$:
$$
\range(\hat t)\cap \bd [-1,1]^{n+1}
\,\,\,\, \mbox{?`}\hspace{-1.4 mm}\cong ?\,\, \,\,\,
S^{n-1}.
$$


\noindent 
$\mathcal P_n$ is trivially decidable  for $n=1,2$.   
$\mathcal P_3$ is decidable, upon recalling  that   
the
 2-sphere is  only surface
with Euler characteristic 2.  
The decidability of $\mathcal P_4$
 follows from the recognizability of the 3-sphere,
a  highly nontrivial result, \cite{tho}, \cite{pjm}.

 For $n\geq 5$, the literature is of little help to get a decidability result
 for $\mathcal P_n$. 
 Quite the contrary,
  one might conjecture  that $\mathcal P_n$
   has a negative solution, at least for  $n>5$,
  as a consequence of  Novikov's unrecognizability theorem
   for $S^{n-1}$, \,\,     \cite{chelek, sht}. 
  %

And yet,   $\mathcal P_n$
has two special properties: 
{\it    $\range(\hat t)$
   is  automatically constrained to live in  
$(n+1)$-space by the parametric map 
$\hat t=(\hat t_1,\dots,\hat t_{n+1})$, and
its intersection with $\bd[-1,1]^{n+1}$
is  to be compared with $S^{n-1}$}. 
%
 %
%
%
%

For Problem \ref{problem:final} and its variants,
knowledge of the dimension $r$ of the ambient space where  
$\range(\hat t)$ 
lives,  and of the difference between
$r$  and the dimension 
of the sphere to be compared with 
$\range(\hat t)\cap \bd[-1,1]^r$,  may 
turn out to be decisive.
 Thus the dovetailing construction
  in the proof of  
 {Theorem  \ref{theorem:main}}  yields
  the {\it decidability} of the problem
\begin{equation}
\label{equation:ultima}
\range(\hat t)\cap \bd[-1,1]^n \,\,\,\,\mbox{?`}\hspace{-1.4 mm}\cong ?\,\, 
  \,\,S^{n-1},
  \end{equation}
  for all $n$,      independently of 
   the dimension $m$
   of the domain of the parametrization  $\hat t=(\hat t_1,\dots,\hat t_n)$
   of $\range(\hat t)$. 
   On the other hand,  the geometric counterpart of 
 {Theorem \ref{theorem:main-due}} states the 
 {\it undecidability}  of 
 \eqref{equation:ultima}  when  $S^{n-1}$  is replaced by an
arbitrary  lower-dimensional sphere  $S^l$.
\end{remark}


 \bibliographystyle{plain}

\begin{thebibliography}{2}
 
\bibitem{and}
M. Anderson, T. Feil,  Lattice-ordered groups.
An Introduction, D. Reidel, Dordrecht, 1988.
  
  \bibitem{bak}
K. A. Baker,
Free vector lattices,
{Canad. J. Math.},
{20} (1968)  58--66.


 



\bibitem{bey75}
W. M. Beynon,
Duality theorems for finitely generated
vector lattices,
{Proc. London Math. Soc.  (3)},
{ 31} (1977)  238--242.

\bibitem{bey77due}
W. M. Beynon,
Applications of duality in the theory of
finitely generated lattice-ordered abelian groups,
{ Canad. J. Math.},
{29} (1977) 243--254.

\bibitem{bkw}
{A. Bigard, K. Keimel, S. Wolfenstein}, {Grou\-pes et
Anneaux R\'{e}ticul\'{e}s}, vol. 608 of {Lecture Notes in
Mathematics},  Springer-Verlag, Berlin, 1971.

%
%
 
 
 

%
%
%
%
%
%
%
%
%


\bibitem{chelek}
A.V. Chernavsky,  V.P. Leksine,
Unrecognizability of manifolds, 
Annals of Pure and Applied Logic, 141 (2006) 325--335.



\bibitem{eva}
T. Evans,
Finitely presented loops, lattices, etc.
are hopfian,  J. London Math. Soc.,   
44 (1969) 551--552.


\bibitem{ewa}
{G. Ewald},
      {Combinatorial convexity and algebraic geometry},
       Graduate Texts in Mathematics, Vol. 168, 
      Springer-Verlag,
      Berlin, Heidelberg,
      1996.




\bibitem{gla-book}
A.M.W. Glass, 
Partially Ordered Groups, in: Series in Algebra, vol. 7, World Scientific Pub. Co., Singapore, 1999.

\bibitem{glahol}
A.M. Glass,W.C. Holland, (Eds.),
Lattice-Ordered Groups: Advances and Techniques,
Mathematics and its Applications,  vol. 48,
Kluwer Academic Publishers, Dordrecht, 1989.

\bibitem{glamad}
A. M. W. Glass, J. J. Madden, The word problem versus the
isomorphism problem,   
J. London Math. Soc.,  
30 (1984) 53--61.


 
\bibitem{harwri}
G. H. Hardy, E. M. Wright, An introduction to the theory of numbers,
(5th Edition), Oxford, Clarendon Press, 1960.





\bibitem{hir}
 R. Hirshon, 
Some Theorems on Hopficity,
Transactions of the American Mathematical Society,  141
(1969)   229--244.

 
 \bibitem{kos}
A.I. Kostrikin, I.R. Shafarevich,
Algebra II: Noncommutative Rings Identities,
Springer, New York, 2012.

%
%

\bibitem{pjm}
A.  Mijatovi\'{c},
Simplifying triangulations of $S^3$,
Pacific
Journal of
Mathematics,
208.2 (2003)
291--324.



%
%
%


\bibitem{mun-tams}
D. Mundici,
Simple Bratteli diagrams with a G\"odel incomplete
C$^*$-equivalence problem, Transactions of the
American Mathematical Society,
 356. 5  (2003)  1937--1955.
 


\bibitem{oda} 
 T. Oda, 
 Convex Bodies and Algebraic Geometry, 
 Springer-Verlag, Berlin, 1988.
 

\bibitem{ams}
E. Outerelo,  J. M. Ruiz,
Mapping Degree Theory, 
Graduate Studies
in Mathematics
Vol. 108,
American Mathematical Society and
Real Sociedad Matem\'atica Espa$\tilde{\rm n}$ola,
Providence, RI, Madrid, 2009. 
 

\bibitem{sht}
M. A. Shtan'ko, Markov's theorem
and algorithmically non-recognizable combinatorial manifolds,
{Izvestiya RAN, Ser.Math.}, 
 68 (2004) 207--224.
%


\bibitem{sta}
J. R. Stallings,
Lectures on
Polyhedral Topology,
Tata Institute of Fundamental Research, Mumbay
1967.

\bibitem{tho}
A. Thompson,  
Thin position and the recognition problem for $S^3$,  
Mathematical Research Letters,
1.5   (1994) 613--630. 

 





\bibitem{wei}
V. Weispfenning,
The complexity of the word problem for
abelian $\ell$-groups,
Theoretical Computer Science, 48 (1986) 127-132.

 

%





 

\end{thebibliography}

 \end{document}